\documentclass[11pt, leqno]{amsart}

\usepackage{amsmath,amssymb,amsthm,bbm}
\usepackage{graphicx}
\usepackage{subfigure}
\usepackage{enumerate}
\usepackage{url}
\usepackage{booktabs}
\usepackage{color}
\usepackage{thmtools, thm-restate}

\usepackage[
heightrounded,
top=1.15in,
bottom=1.2in,
inner=1.1in,
outer=1.1in
]{geometry} %

\def\arXiv#1{arXiv:\href{http://arXiv.org/abs/#1}{#1}}

\def\MR#1{MR\href{http://www.ams.org/mathscinet-getitem?mr=#1}{#1}}

\def\C{\mathbb C}

\def\R{\mathbb R}
\def\RP{\mathbb {RP}}

\def\S{\mathbb S}

\def\theta{\vartheta}

\usepackage{url}

\usepackage[colorlinks,linkcolor=blue,citecolor=magenta,urlcolor=black,hypertexnames=false]{hyperref}

\newtheorem{theorem}{Theorem}
\numberwithin{theorem}{section}
\newtheorem{conj}{Conjecture}
\numberwithin{conj}{section}
\newtheorem{corollary}{Corollary}
\numberwithin{corollary}{section}
\newtheorem{lemma}{Lemma}[section]
\numberwithin{lemma}{section}
\newtheorem{proposition}{Proposition}
\numberwithin{proposition}{section}

\title{Repeated minimizers of $p$-frame energies}
\author{Alexey Glazyrin}
\author{Josiah Park}
\address{School of Mathematical and Statistical Sciences, The University of Texas Rio Grande Valley, Brownsville, TX 78520}
\email{alexey.glazyrin@utrgv.edu}
\address{School of Mathematics, Georgia Institute of Technology, Atlanta, GA 30332-0160}
\email{j.park@gatech.edu}
\date{\today}

\begin{document}

\maketitle
\begin{abstract}
For a collection of $N$ unit vectors $\mathbf{X}=\{x_i\}_{i=1}^N$, define the $p$-frame energy of $\mathbf{X}$ as the quantity $\sum_{i\neq j} |\langle x_i,x_j \rangle|^p$. In this paper, we connect the problem of minimizing this value to another optimization problem, so giving new lower bounds for such energies. In particular, for $p<2$, we prove that this energy is at least $2(N-d) p^{-\frac p 2} (2-p)^{\frac {p-2} 2}$ which is sharp for $d\leq N\leq 2d$ and $p=1$. We also prove that for $1\leq m<d$, a repeated orthonormal basis construction of $N=d+m$ vectors minimizes the energy over an interval, $p\in[1,p_m]$, and demonstrate an analogous result for all $N$ in the case $d=2$. Finally, in connection, we give conjectures on these and other energies.
\end{abstract}

\section{Introduction}\label{sec:intro}

Let $\mathbf{A}=A_{i,j}$ be an $N\times N$ real matrix of rank less than or equal to $d$, and with ones along the diagonal. The $p$-frame energy of matrix $\mathbf{A}$ is denoted

$$E_p(\mathbf{A})=\sum\limits_{i\neq j} |A_{ij}|^p.$$


An interesting question is what the optimizing matrices for $E_p(\mathbf{A})$ are for fixed $p$, $N$, and $d$. Bukh and Cox in \cite{buk18} recently studied the question of bounding $E_{\infty}(\mathbf{A})=\max\{|A_{ij}|\}$ and its consequences. One special case of this problem concerns matrices associated with unit vectors, $\mathbf{X}=\{x_i\}_{i=1}^N\subset \R^d$, in which case $\mathbf{A}$ is the Gram matrix of $\mathbf{X}$ and so is additionally symmetric and positive semi-definite.

When $p$ is an even natural number, and $N$ is sufficiently large, sharp bounds for energies in the  real and complex case follow from the work of Sidelnikov and Welch \cite{sid74,wel74}. Several related results and conjectures on minimizers of $p$-frame energies were formulated by Ehler and Okoudjou in \cite{ehl12}. Using the fact that for $d=N$, the minimizer of $E_p(\mathbf{A})$ for $p=2$ is an orthonormal basis, they show that whenever $N$ is divisible by $d$ and $p\in(0,2)$, a repeated orthonormal basis is the unique minimizer. Since the value of the $p$-frame energy does not change when substituting a vector by its opposite, uniqueness in this scenario is meant up to symmetry. Alternatively, uniqueness can be understood by considering the energy on projective spaces $\RP^{d-1}$ (as is done in \cite{bil19a}), where the points in $\RP^{d-1}$ may be identified with lines through the origin in $\R^d$.

The problem of minimizing the 1-frame energy was also posed in \cite{mat17}, where it was conjectured that for any $N$, the repeated orthonormal basis is the unique minimizer. In 1959, Fejes T\'{o}th posed the question \cite{fej59}: what is the largest sum of non-obtuse angles formed by $N$ lines in $\mathbb{R}^d$? The conjecture stands that for any $N$ the maximum is uniquely attained on a collection of $d$ lines generated by a repeated orthonormal basis and is resolved only for $d=2$, and for very few cases of $N$ for $d=3$. The asymptotic result for $d=3$ is wide open (see \cite{bil19} and \cite{fod16} for recent progress). We note that the conjecture about $E_1$ from \cite{mat17} immediately follows from the conjecture of Fejes T\'{o}th since $\arccos t \geq \frac {\pi} 2 (1-t)$ for $t\in[0,1]$ with equality holding precisely at $t=0$ and $t=1$.

In this paper, we develop new methods for finding lower bounds on $E_p(\mathbf{A})$, based on the framework of Bukh and Cox from \cite{buk18}. Doing so allows us to prove new general bounds for $E_p(\mathbf{A})$ when $p<2$. Such bounds are sharp in some cases, particularly, for $p=1$ and $N\in [d+1, 2d]$. We also give sharp bounds for $N=d+m$ and $p\in[1,2\log\frac{2m+1}{2m}/\log\frac{m+1}{m}]$, thus, in particular, partially confirming a conjecture from \cite{che19}.

Although everything in Sections \ref{sec:aux}-\ref{sec:new} is formulated in the real case, all observations and proofs there work for matrices of complex numbers or quaternions without any changes. Our methods work for general matrix optimization problems, so most of our results will be formulated for matrices. However, we slightly abuse terminology when talking about vector sets instead of their Gram matrices while speaking of $E_p(\mathbf{A})$.

In Section \ref{sec:2dim}, we prove that the $p$-frame energy for unit vectors in the plane is minimized by repeated orthonormal bases for any number of vectors if $p\in[1,1.3]$. In Section \ref{sec:disc} we discuss possible generalizations of the results of the paper and motivations behind them.

\section{Auxiliary problems and tight frames}\label{sec:aux}

Bukh and Cox in \cite{buk18} introduced a method for deriving new packing bounds for projective codes. In our related approach we use the notion of a {\it tight frame}. A tight frame is a finite collection of vectors $\{y_i\}_{i=1}^N$ in $\R^d$ with the property that,
\begin{equation}\label{eq:tightframe} \sum\limits_{i=1}^N \langle x, y_i \rangle^2=A \|x\|^2,  \end{equation}
holds, for any $x\in\mathbb{R}^d$ and some $A>0$. $A$ is called a {\it frame constant} of the tight frame.

Using the tight frame condition (\ref{eq:tightframe}) and comparing coefficients for all $d$ components of $x$, one can conclude that $\sum_{i=1}^N \|y_i\|^2=Ad$. It is convenient to use the normalization $A=\frac{1}{d}$ for frames as above, so that the Hilbert-Schmidt norm of the $d\times N$ matrix $\mathbf{Y}$, with column vectors $\{ y_{i} \}_{i=1}^N$, is normalized to be $1$, 
\begin{equation}\label{eq:normalize} \|\mathbf{Y}\|^2_{\text{HS}}= \sum\limits_{i=1}^N \|y_i\|^2=1. \end{equation}
In the next two lemmas we collect instruments for computing new lower bounds for discrete $p$-frame energies. The first makes a connection between kernels of matrices and associated tight frames. We introduce the notation $f_{c,p}(t)=\left(\frac{t}{c-t} \right)^{\frac{p}{2}}$, to be used in the second lemma. We also introduce the optimization problem,
\begin{equation}\label{eq:mcpn} M(c,p,N)=\min\left\{\sum_{i=1}^N f_{c,p}(t_i)\right\vert\left. \sum_{i=1}^N t_i = 1,\ t_i\in[0,c)\right\}, \end{equation}
where $p>0$ and $c>1/N$. Clearly, this optimization problem is properly defined.

\begin{lemma}\label{lem:frame}
For any real $N\times N$ matrix $\mathbf{A}$ of rank $d$, $N\geq d+1$, with unit diagonal elements, there exists a tight frame $\{y_1, y_2,\ldots, y_N\}\subset\mathbb{R}^{N-d}$ with the frame constant $\frac 1 {N-d}$ such that $\text{Ker\,}\mathbf{A}$ consists of all vectors of the form $(\langle y, y_1\rangle,\ldots,\langle y, y_N\rangle )$ with $y\in\mathbb{R}^{N-d}$.  
\end{lemma}

\begin{proof}
$\text{Ker\,}\mathbf{A}$ is $(N-d)$-dimensional so there is a linear mapping $L:\mathbb{R}^{N-d}\rightarrow\mathbb{R}^{N}$ whose image is $\text{Ker\,}\mathbf{A}$. For each of $N$ components, $L$ is a linear functional so it may be represented as $L_i (y)=\langle y,z_i \rangle$. We note that for any non-singular mapping $D:\mathbb{R}^{N-d}\rightarrow\mathbb{R}^{N-d}$, the image of the mapping $L\circ D$ is $\text{Ker\,}\mathbf{A}$ as well. The quadratic form $\sum_{i=1}^N \langle y,z_i\rangle^2$ is positive definite, and so by choosing a suitable $D$, we can transform $\{z_1,z_2,\ldots,z_N\}$ into $\{y_1,y_2,\ldots,y_N\}$ so that $\sum_{i=1}^N \langle y,y_i\rangle^2 = \frac 1 {N-d} \langle y,y \rangle$. So, we obtain the condition (\ref{eq:tightframe}), and $\{y_i\}_{i=1}^N$ is a tight frame.
\end{proof}

The construction in Lemma \ref{lem:frame} is due to Bukh and Cox \cite{buk18} who used it for obtaining new packing bounds for projective codes. It also can be interpreted as a tight frame representative of a Gale dual to the matrix $A$ (see \cite{gla19} for more details about this interpretation).

\begin{lemma}\label{lem:min}
For any real $N\times N$ matrix $\mathbf{A}$ of rank $d$, $N\geq d+1$, with unit diagonal elements,
$$E_p(\mathbf{A})\geq M\left(\frac 1 {N-d}, p, N\right)\ \text{ if }\ 1\leq p\leq 2,$$
$$E_p(\mathbf{A})\geq (N-1)^{1-\frac p 2} M\left(\frac 1 {N-d}, p, N\right)\ \text{ if }\ p\geq 2$$
\end{lemma}

\begin{proof}
By Lemma \ref{lem:frame}, there exists a tight frame $\{y_1, y_2,\ldots, y_N\}\subset\mathbb{R}^{N-d}$ such that $\text{Ker\,}\mathbf{A}$ is the set of all vectors $(\langle y, y_1\rangle,\ldots,\langle y, y_N\rangle )$ for some $y\in\mathbb{R}^{N-d}$ and $\sum_{i=1}^N |y_i|^2 = 1$. Taking $y=y_i$ and using the kernel condition for row $i$, we get

$$\langle y_i, y_i \rangle + \sum_{j\neq i} A_{ij} \langle y_i, y_j \rangle = 0.$$

Then for any $1\leq i \leq N$,

$$ \langle y_i, y_i \rangle \leq \sum_{j\neq i} \vert A_{ij}\vert \vert\langle y_i, y_j \rangle\vert \leq \left( \sum_{j\neq i}|A_{ij}|^p\right)^{\frac 1 p} \left( \sum_{j\neq i}\vert\langle y_i, y_j \rangle\vert^q\right)^{\frac 1 q},$$
by H\"{o}lder's inequality for $\frac 1 p + \frac 1 q =1$ (for $q=\infty$, $\left(\sum_{j\neq i}\vert\langle y_i, y_j \rangle\vert^q\right)^{\frac 1 q}=\max_{j\neq i}\vert\langle y_i, y_j \rangle\vert$).

By monotonicity of norms $\|\cdot \|_p$ in $p$ and H\"{o}lder's inequality (for vectors in $\mathbb{R}^{N-1}$), $\Vert x \Vert_q \leq \Vert x \Vert_2$ for $q\geq 2$, while $\Vert x \Vert_q \leq (N-1)^{\frac 1 q - \frac 1 2}\Vert x \Vert_2$ when $q\leq 2$. Hence,

$$\langle y_i, y_i \rangle \leq \left( \sum_{j\neq i}|A_{ij}|^p\right)^{\frac 1 p} \left( \sum_{j\neq i}\langle y_i, y_j \rangle^2\right)^{\frac 1 2}\ \text{ if }p\leq 2,\text{ and},$$

$$\langle y_i, y_i \rangle \leq (N-1)^{\frac 1 2 - \frac 1 p} \left( \sum_{j\neq i}|A_{ij}|^p\right)^{\frac 1 p} \left( \sum_{j\neq i}\langle y_i, y_j \rangle^2\right)^{\frac 1 2}\ \text{ if }p\geq 2.$$

At this point we use the tight frame condition for $y_i$, i.e. $\sum_{j\neq i}\langle y_i, y_j \rangle^2 = \frac 1 {N-d} \langle y_i, y_i \rangle - \langle y_i, y_i \rangle^2$, and denote $\langle y_i, y_i \rangle$ by $t_i$ to arrive finally at the inequalities:

$$ \left( \sum_{j\neq i}|A_{ij}|^p\right)^{\frac 1 p} \geq \left(\frac {t_i} {\frac 1 {N-d} - t_i} \right)^{\frac 1 2}\ \text{ if }p\leq 2,\text{ and},$$

$$ \left( \sum_{j\neq i}|A_{ij}|^p\right)^{\frac 1 p} \geq (N-1)^{\frac 1 p - \frac 1 2} \left(\frac {t_i} {\frac 1 {N-d} - t_i} \right)^{\frac 1 2}\ \text{ if }p\geq 2.$$

Taking powers, noting $\sum_{i=1}^N t_i=1$, and summing these inequalities over all $i$, we obtain the conclusion of the lemma.
\end{proof}

\section{New lower bounds for the $p$-frame energy}\label{sec:new}

As a first application of Lemma \ref{lem:min}, we give a new proof of the result from \cite{okt07} (also Proposition 3.1 in \cite{ehl12}).

\begin{proposition}\label{prop:p>2}
For any $p\geq 2$ and any real $N\times N$ matrix $\mathbf{A}$ of rank $d$, $N\geq 2$, with unit diagonal elements,
$$E_p(\mathbf{A})\geq N(N-1) \left( \frac {N-d} {d(N-1)} \right)^{\frac p 2}.$$
\end{proposition}

\begin{proof}
For $N=d$, the right-hand side is 0. We assume $N\geq d+1$ for the rest of the proof.

For $p\geq 2$, $f_{c,p} (t) =\left(\frac {t} {c-t}\right)^{\frac p 2}$ is convex on $[0,c)$. Due to Jensen's inequality, this implies 

$$M\left(\frac 1 {N-d}, p, N\right) = N f_{\frac 1 {N-d},p} \left(\frac 1 N\right) = N \left(\frac {N-d} {d}\right)^{\frac p 2}.$$

Together with Lemma \ref{lem:min} this completes the proof.
\end{proof}

It is straightforward to check that when $p>2$, the only case in which the inequality of Proposition \ref{prop:p>2} is exact, holds when $\mathbf{Y}=\{y_1,\ldots,y_N\}$ is a tight frame in $\mathbb{R}^{N-d}$ which satisfies that $|y_i|$ is constant for all $i$ and $\vert\langle y_i,y_j\rangle\vert$ is also constant for all $i\neq j$. In other words equality holds in Proposition \ref{prop:p>2} if and only if $\mathbf{Y}$ is an \textit{equiangular tight frame} (ETF). In view of $\mathbf{Y}$ as an ETF, matrix $\mathbf{A}$ is then the Gram matrix of the $d$-dimensional ETF known as the \textit{Naimark complement} or \textit{Gale dual} to $\mathbf{Y}$ (see, for instance, \cite{cas13} and \cite{coh16} for more details about Naimark complements and Gale duality of equiangular tight frames).

There are several interesting properties of ETFs but the two most fundamental are that they are precisely the equality achieving systems of vectors for the {\it Welch bound}, and the maximum size $N$ of such systems is limited by {\it Gerzon's bound}.

The Welch bound gives a lower bound for the {\it coherence} of a system of unit vectors namely, for unit vectors $\{\varphi_{i}\}_{i=1}^N\subset \R^{d}$ or $\C^{d}$,
$$\max\limits_{i\neq j} |\langle \varphi_i,\varphi_j \rangle|\geq \sqrt{\frac{N-d}{d(N-1)}}.$$
This bound is one example of several other bounds which limit how spread out the one-dimensional subspaces corresponding to each vector may be \cite{wel74}; see \cite{lev82} for similar bounds and their derivation from a linear programming approach.

Gerzon's bound \cite{lem73} limits the size of an ETF, and states in the real and complex case respectively,

$$N\leq \binom{d+1}{2}\ \text{ and }\ N\leq d^2.$$

We call ETFs attaining this bound {\it maximal}. Maximal ETFs in real spaces $\R^d$ may exist only when $d=1, 2, 3$ or $(2m+1)^2-2$ for some natural $m$. There are known maximal ETFs for $d=1, 2, 3, 7, 23$ only. More details on the existence of maximal ETFs and the improvement of Gerzon's bound for $d\neq (2m+1)^2-2$ are available in \cite{gla18}. The existence of complex maximal ETFs is a famous open conjecture \cite{zau99}. For maximal ETFs, the Welch bound must hold as well so their coherence must satisfy $\alpha^2=\frac 1 {d+2}$ in the real case and $\alpha^2=\frac 1 d$ in the complex case. We discuss maximal ETFs further in connection with Theorem~\ref{thm:cont-energy} in Section~\ref{sec:disc}.

Using the lemmas from the previous section, we now give another observation on the relation between optimizing $E_p(\mathbf{A})$ and the problem $M(c,p,N)$.

\begin{theorem}\label{thm:gen}
For any $1\leq p< 2$ and any real $N\times N$ matrix $\mathbf{A}$ of rank $d$ with unit diagonal elements,
$$E_p(\mathbf{A})\geq \frac {2(N-d)}{p^{\frac p 2} (2-p)^{\frac {2-p} 2}}.$$
\end{theorem}

\begin{proof}
For $N=d$, the right-hand side is 0. We assume $N\geq d+1$ for the rest of the proof.

$f_{\frac 1 {N-d}, p} (t) / t$ is minimized on $(0,\frac 1 {N-d})$ at $t_0=\frac {2-p} {2(N-d)}$. Hence, $$f_{\frac 1 {N-d}, p} (t)\geq \frac {2(N-d)} {p^{\frac p 2} (2-p)^{\frac {2-p} 2}} t,\ \text{ and so,}$$

$$M\left(\frac 1 {N-d}, p, N\right)\geq \frac {2(N-d)} {p^{\frac p 2} (2-p)^{\frac {2-p} 2}} \sum_{i=1}^N t_i = \frac {2(N-d)} {p^{\frac p 2} (2-p)^{\frac {2-p} 2}}.$$ The final bound then follows from Lemma \ref{lem:min}.
\end{proof}

When taking $p=1$ in Theorem \ref{thm:gen}, we get $E_1(\mathbf{A})\geq 2(N-d)$. For $N$ in the range $d+1\leq N \leq 2d$, we thus obtain the bound conjectured in \cite{mat17} from Theorem \ref{thm:gen}. We formulate it here as a simple statement about angles between lines in Euclidean spaces where it is understood that such angles are restricted to lie in $[0,\pi/2]$. 

\begin{corollary}\label{cor:1-frame}
The sum of cosines of all pairwise angles between $N$ lines in $\mathbb{R}^d$ is at least $N-d$. For $N\in[d,2d]$, the bound is sharp and the unique minimizer is the set of $N$ lines forming a repeated orthonormal basis.
\end{corollary}

As hinted in the discussion \cite{mat17}, Corollary \ref{cor:1-frame} may be proven by induction. For the sake of completeness, we provide such a proof here as well.

\begin{proof}[Alternative proof of Corollary \ref{cor:1-frame}]
We choose a unit vector in the direction of each of the $N$ lines and construct an $N\times N$ Gram matrix for the chosen vectors. The matrix has rank no greater than $d$ so, by the Gershgorin circle theorem, any $(d+1)\times(d+1)$ diagonal minor of the matrix will have at least one row whose sum of absolute values of non-diagonal entries is at least 1. The inductive step consists then in finding one row like this and using the inductive hypothesis for the $(N-1)\times (N-1)$ diagonal minor obtained by deleting this row and the column symmetric to it from the matrix. \end{proof}

We do not know how to extend this short proof to a more general problem of finding lower bounds for the 1-frame energy of matrices that is covered by Theorem \ref{thm:gen}. The proof of Corollary \ref{cor:1-frame} does not seem to work for non-symmetric matrices. We also note that Theorem \ref{thm:gen} implies the same bound $2(N-d)$ for $E_p(\mathbf{A})$ when $p\in(0,1)$.

For the case of $N=d+1$, Chen, Gonzales, Goodman, Kang, and Okoudjou \cite{che19} posed a conjecture for the minimum of the $p$-frame energy for all $p\in(0,2)$. They conjectured that a global minimum is necessarily formed by $k+1$ unit vectors whose endpoints form a regular $k$-dimensional simplex and $N-k-1$ vectors that are pairwise orthogonal and orthogonal to the subspace of the regular simplex. In particular, their conjecture states that for $p<\frac {\ln 3} {\ln 2}\approx 1.58496$, the minimum is 2 and attained only on the repeated orthogonal basis with $d+1$ vectors.

We study a more general problem of $d+m$ vectors, $1\leq m < d$, and, using Lemma \ref{lem:min}, prove that the repeated orthonormal basis minimizes $E_p$ for $p
\in [1,p_m]$. In particular, for $m=1$, our results confirm the conjecture from \cite{che19} for $p$ in the range $p\leq 2(\frac {\ln 3} {\ln 2}-1)\approx 1.16993$.

\begin{theorem}\label{thm:d+1}
For $p\in[1,2\log{\frac{2m+1}{2m}}/\log{\frac{m+1}{m}}]$, $1\leq m< d$ and a real $(d+m)\times (d+m)$ matrix $\mathbf{A}$ of rank $d$ with ones along the diagonal, 
$$E_{p}(\mathbf{A})=\sum\limits_{i\neq j}|A_{i,j}|^{p} \geq 2m.$$
\end{theorem}

The following lemma is used in the proof of Theorem \ref{thm:d+1}.

\begin{lemma} Set $\alpha=\frac 1 m \left(\frac{1}{2}-\frac{p}{4}\right)$. For $p\in[1,2]$, $M(\frac{1}{m},p,N)$ is minimized for $t_{j}$ of the form
\begin{enumerate}
\item[(i)] $t_1=\dots=t_{k}=\frac{1}{k},\ t_{k+1}=\dots=t_{n}=0$, where $\frac{1}{k}\geq \alpha$, or,  

\item[(ii)] $t_1=\dots=t_k=x$, $t_{k+1}=1-kx$, $t_{k+2}=\dots=t_{N}=0$, where $\alpha\leq x < \frac 1 m$, $0<1-kx<\alpha$.
\end{enumerate}
\end{lemma}
\begin{proof}
Computing the second derivative of $f_{1/m,p}(t)$, we see it is concave on $[0,\alpha]$ and convex on $[\alpha,\frac 1 m)$ where $\alpha=\frac 1 m \left(\frac{1}{2}-\frac{p}{4}\right)$. All $t_i$ lying in $[0,\alpha]$ may be moved to the endpoints of the interval, except for at most one number, while keeping their sum constant and minimizing the sum of values of $f_{1/m,p}$ (this follows, for instance, from the Karamata inequality, see \cite[pg 89]{har88}). After this we may apply Jensen's inequality for all numbers from $[\alpha,\frac 1 m)$ and assume they are all equal. The resulting minimizer is then necessarily one of the two types: 1) $t_1=\ldots=t_k=\frac 1 k$, $t_{k+1}=\ldots=t_N=0$, where $\frac 1 k \geq \alpha$, 2) $t_1=\ldots=t_k=x$, $t_{k+1}=1-kx$, $t_{k+2}=\ldots=t_N=0$, where $x\geq\alpha$ and $0<1-kx<\alpha$. \end{proof}

We now follow with a proof of Theorem \ref{thm:d+1}.

\begin{proof}
\indent Set $p_m=2\log{\frac{2m+1}{2m}}/\log{\frac{m+1}{m}}$ and $q_{m}=\frac{p_m}{2}$. Clearly, it is sufficient to prove the lower bound for $p=p_m$ only. We use Lemma \ref{lem:min} and show that $M(\frac{1}{m},p,N)\geq 2m$. Consider the first case in the above lemma, so that $$t_1=\dots=t_{k}=\frac{1}{k},\ t_{k+1}=\dots=t_{n}=0,$$ where $\frac{1}{k}\geq \alpha=\frac 1 m \left(\frac{1}{2}-\frac{p_m}{4}\right)$. In this case, we need to minimize the value $$kf_{1/m,p_m}\left(\frac{1}{k}\right)=k\left( \frac {m} {k-m} \right)^{\frac{p_m}{2}}.$$
 The real function \begin{equation}\label{eq:fm} F_m(x)=x\left( \frac {m} {x-m} \right)^{\frac{p_m}{2}} \end{equation} for $x>m$ has exactly one local minimum. The degree $p_m$ was specifically chosen so that $F_m(2m)=F_m(2m+1)=2m$. Then by Rolle's theorem, the local minimum of $F_m(x)$ lies in $[2m,2m+1]$. The minimum of $F_m(x)$ for natural values of $x$, $x>m$, is, therefore, attained on $2m$ and $2m+1$ and is equal to $2m$.

\indent In the second case, $x<\frac{1}{k}$ and $x\geq \alpha=\frac 1 m \left(\frac{1}{2}-\frac{p_m}{4}\right)$. It is straightforward to show that $p_m<\frac {4m+2} {4m+1}$ for all natural $m$. Subsequently $k<\frac 1 {\alpha} < 4m+1$ so that $k$ can take (integer) values only in $[m,4m]$. To show $E_{p_m}\geq 2m$ it suffices then to show for all $m\leq j\leq 4m$, and all $x$ in $I=(\frac{1}{j+1},\frac{1}{j})$ that the function 
\begin{align*}    g_{j}(x) =  j \left( \frac{mx}{1-mx} \right)^{q_{m}}+\left(\frac{m(1-jx)}{1-m(1-jx)}\right)^{q_{m}},  \end{align*} 
satisfies $g_{j}(x)\geq 2m$. This will be demonstrated using properties specific to $g_{j}(x)$, namely that the function has at most one critical point, $g_{j}^{\prime}(x)=0$, inside the interval $I$. Taking derivatives,
 \begin{align*} g^{\prime}_{j}(x)&=q_{m} jm \left[ \left( \frac{mx}{1-mx}  \right)^{q_{m}-1} \frac{1}{(1-mx)^2}- \left(\frac{m(1-jx)}{1-m(1-jx)}\right)^{q_{m}-1} \frac{1}{(1+m(-1+jx))^2} \right], \end{align*}

 so that $g'_{j}(x)=0$  gives
 \begin{align*}  \left( \frac{x(1+m(-1+jx))}{(1-mx)(1-jx)} \right)^{q_{m}-1} &= \frac{(1-mx)^2}{(1+m(-1+jx))^2} \\ 
\left( \frac{x(1+m(-1+jx))}{(1-mx)(1-jx)} \right)^{q_{m}+1} &=  \frac{x^2}{(1-jx)^2}  \\
  \frac{1+m(-1+jx)}{1-mx}  &= \left( \frac{x}{(1-jx)} \right)^{\frac{2}{q_{m}+1}-1} \\
 \frac{1-mx}{1+m(-1+jx)}  &= \left( \frac{x}{(1-jx)} \right)^{1-\frac{2}{q_{m}+1}}.
\end{align*} %
\vspace{1 mm}
Calling the function on the left in the above expression $f(x)$ and the function on the right $g(x)$,
\begin{align*}
f^{\prime \prime}(x) = \frac{2j(1+j-m)m^2}{(1+m(-1+jx))^3}>0\ \text{ on }\ I, \end{align*} while letting $\gamma=1-\frac{2}{q_{m}+1}$, \begin{align*} g^{\prime \prime}(x)  = \frac{\gamma (\frac{x}{1-jx})^\gamma (\gamma-1+2jx)}{x^2(jx-1)^2}<0\ \text{ on }\ I, \end{align*}   since $\gamma<0$. Thus $f(x)$ is convex on $I$, while $g(x)$ is concave on $I$. Since $f(\frac{1}{j+1})=g(\frac{1}{j+1})=1$, it must be the case then that $f(x)=g(x)$ for no more than one point $x\in I$, ($x\neq \frac{1}{j+1},\frac{1}{j}$). Now, \begin{align*} g_{j}'\left(\frac{1}{j+1}\right)=0\ \text{ and }\ \lim\limits_{x\rightarrow \frac{1}{j}} g_{j}'\left(x\right)=-\infty. \end{align*} \indent Thus if there is a critical point, it corresponds to a local maximum of $g_{j}(x)$ and it suffices to check the value of $g_{j}(x)$ at the endpoints in $I$ for each $m\leq j\leq 4m$ to establish the desired lower bound. These values are
\begin{align*} g_{j}\left(\frac{1}{j+1}\right)=(1+j)\left(\frac{m}{1+j-m}\right)^{q_m}=F_m(j+1),\ \text{ and }\  g_{j}\left(\frac{1}{j}\right)=j\left(\frac{m}{j-m}\right)^{q_m}=F_m(j), \end{align*}
where $F_m$ is the function defined in equation (\ref{eq:fm}). The minimal value of $F_m(x)$ on natural numbers, $x>m$, as we established earlier, is precisely $2m$. \end{proof}

Following the proofs of Theorem \ref{thm:d+1} and Lemma \ref{lem:min} it is easy to check that the only minimizer is the repeated orthonormal basis. Theorem \ref{thm:d+1} was first announced by the second author in \cite{par19}.

\section{Minimizing energy in two dimensions}\label{sec:2dim}

In this section, we study the problem of minimizing the $p$-frame energy for collections of unit vectors in the plane. In particular, we show that the repeated orthonormal basis is the minimizer for $p\leq 1.3$. As one of the instruments for our proof, we use the solution of the Fejes T\'{o}th problem mentioned in Section \ref{sec:intro}.

\begin{theorem}\label{thm:fej-2}
Let $x_1, x_2, \ldots, x_N$ be unit vectors in the plane. Then, $$\sum\limits_{i,j=1}^N \arccos|\langle x_i,x_j \rangle|\leq \frac{\pi N^2 }{4}, \text{ if } N \text{ is even},$$ $$\sum\limits_{i,j=1}^N \arccos|\langle x_i,x_j \rangle|\leq  \frac{\pi (N^2-1)}{4}, \text{ if } N \text{ is odd}.$$
\end{theorem}

Theorem \ref{thm:fej-2} was proven in \cite{fod16}. Several alternative proofs were also obtained in \cite{bil19}.

\begin{theorem}\label{thm:dim-2}
Let $\mathbf{A}$ be a Gram matrix of $N$ unit vectors in the plane. Then for $p\in (0,1.3]$ , $E_p(A)\geq N(N-2)/2$ if $N$ is even and $E_p(A)\geq (N-1)^2/2$ if $N$ is odd.
\end{theorem}

\begin{proof}
Assume $\mathbf{A}$ is the Gram matrix of unit vectors $x_1, x_2, \ldots, x_N$ in the plane.

For any $p\in [0,2]$, the lower bound on $E_p$ for even $N$ follows immediately from the fact that a repeated orthonormal basis is one of the minimizers of $E_2$ when the number of vectors is divisible by the dimension (see \cite{ehl12,sid74,ven01}).

It is sufficient to prove the lower bound for $p=1.3$ so we consider this case only for the rest of the proof. For odd $N$, we split our proof into two parts: 1) angles between each pair of vectors are sufficiently far from $\pi/2$; 2) there are vectors that are almost orthogonal. For the first case, we assume that all angles $\arccos|\langle x_i, x_j\rangle |$, $1\leq i,j\leq N$, are no greater than $1.34$. It is straightforward to check that for any $t\in[0,1.34]$,
$$\cos^{1.3} t \geq \frac 2 {\pi} \left(\frac \pi 2 - t\right).$$

Summing these inequalities for all $t=\arccos|\langle x_i, x_j\rangle|$, $1\leq i,j \leq N$, $i\neq j$, we obtain

$$E_{1.3} (\mathbf{A})\geq N^2 - N - \frac 2 \pi \sum\limits_{i,j=1}^N \arccos|\langle x_i,x_j \rangle|.$$

Using the solution of the Fejes T\'{o}th problem from Theorem \ref{thm:fej-2} we conclude

$$E_{1.3} (\mathbf{A})\geq N^2 - N - \frac 2 \pi (N^2-1) \frac \pi 4 = \frac {(N-1)^2} 2.$$

For the second case, we assume that the largest angle among $\arccos|\langle x_i, x_j\rangle |$, $1\leq i,j\leq N$, is at least $1.34$. Without loss of generality, let one of such angles be $\arccos|\langle x_1, x_2\rangle |$. Our proof will be by induction on odd numbers $N$. The statement of the theorem is clearly true for $N=1$. Let $N$ be an odd number greater than 1. Now we will show that for any $i$, $3\leq i\leq N$, $|\langle x_1, x_i\rangle |^{1.3}+|\langle x_2, x_i\rangle |^{1.3}\geq 1$.

We can always switch a vector to its opposite without changing the total energy so we may assume that $x_i$ lies in the angle formed by $x_1$ and $x_2$. We assume the angle between $x_1$ and $x_2$ is $\theta$ and $x_i$ forms the angles of $\phi$ and $\theta-\phi$ with $x_1$ and $x_2$, respectively. Note that both $\phi$ and $\theta-\phi$ must be less than $\pi/2$, otherwise one of them is closer to $\pi/2$ than $\theta$. Without loss of generality, $\phi\leq\theta/2$. There are two possible options: $\theta\leq \pi/2$ and $\theta>\pi/2$.

For the first option, $$|\langle x_1, x_i\rangle |^{1.3}+|\langle x_2, x_i\rangle |^{1.3} = \cos^{1.3} \phi + \cos^{1.3} (\theta-\phi) \geq \cos^{1.3} \phi + \cos^{1.3} \left(\frac \pi 2-\phi\right)  \geq 1.$$


For the second option, $\theta\leq\pi-1.34$ because $\theta=\pi-\arccos|\langle x_1, x_2\rangle |$. The angle $\theta$ is the one closest to $\pi/2$ among all angles formed by the vectors. In particular, $\theta-\phi$ cannot be closer to $\pi/2$ so $\theta-\phi\leq \pi - \theta$. This condition can be rewritten as $\frac {\pi-\phi} 2 \geq \theta-\phi$. For the next step we try to minimize $\cos^{1.3} \phi + \cos^{1.3} \left(\theta-\phi\right)$ by keeping $\phi$ intact and increasing $\theta-\phi$ as much as possible while preserving the conditions $\theta-\phi\leq\pi-\phi-1.34$ and $\theta-\phi\leq \frac{\pi-\phi} 2$. While increasing $\theta$, at some moment we reach the point when one of these two inequalities becomes a precise equality. These two possibilities can be described by the two cases depending on the value of $\phi$.

For the first case, assume $\frac {\pi-\phi} 2 > 1.34$, i.e. $\phi<\pi-2.68$. Then

$$\cos^{1.3} \phi + \cos^{1.3} (\theta-\phi)\geq \cos^{1.3} \phi + \cos^{1.3} \left(\frac {\pi-\phi} 2\right) = \cos^{1.3} \phi + \sin^{1.3} \frac \phi 2.$$
The function $\cos^{1.3} \phi + \sin^{1.3} \frac \phi 2$ is at least 1 for $\phi\in[0, \pi-2.68]$ so the first case is covered.

For the second case, we assume $\frac {\pi-\phi} 2 \leq 1.34$, i.e. $\phi\geq\pi-2.68$. As we assumed earlier $\phi\leq\theta/2$, so $\phi\leq\frac {\pi-1.34} 2$. Using the inequality $\theta<\pi-1.34$ again we get that $\theta-\phi<\pi-1.34-\phi$. This implies

$$\cos^{1.3} \phi + \cos^{1.3} (\theta-\phi)\geq \cos^{1.3} \phi + \cos^{1.3} (\pi-1.34-\phi).$$
The function $\cos^{1.3} \phi + \cos^{1.3} (\pi-1.34-\phi)$ is at least 1 for $\phi\in [\pi-2.68,\frac {\pi-1.34} 2]$. Overall, we conclude that $|\langle x_1, x_i\rangle |^{1.3}+|\langle x_2, x_i\rangle |^{1.3} = \cos^{1.3} \phi + \cos^{1.3} (\theta-\phi)\geq 1.$

Then, by the induction hypothesis for unit vectors $x_3, \ldots, x_N$,

$$E_{1.3} (\mathbf{A})\geq \frac {(N-3)^2} 2 + 2\sum\limits_{i=3}^N (|\langle x_2, x_i\rangle |^{1.3}+|\langle x_1, x_i\rangle |^{1.3})$$

$$ \geq \frac {(N-3)^2} 2 + 2(N-2) = \frac {(N-1)^2} 2.$$ \end{proof}

We do not know how to extend this proof to the general case of matrices of rank 2. The value $p=1.3$ is not the best possible. One can impose all conditions necessary for the proof of Theorem \ref{thm:dim-2} to work and optimize for $p$. The numerical value obtained this way is approximately $1.317$.

\section{Discussion}\label{sec:disc}

Recently, the conjecture mentioned above from the paper \cite{che19} was proved true in \cite{xu19}. In particular, the range that the orthonormal basis plus a vector minimizes in Theorem~\ref{thm:d+1} can be increased to the value $p=\log{3}/\log{2}$. The behavior of the other maximal values of $p$ that similar constructions are expected to minimize for is suggested in the following conjecture (appearing first in \cite{par19}).

\begin{conj} \label{conj:onb} Let $N=m+kd$ points be given in $\mathbb{S}^{d-1}$, with $1\leq m<d$, $d\geq 2$, and a Gram matrix $\mathbf{A}\in\R^{N\times N}$. Then there is a value of $p_0$, independent of dimension $d$ and excess $m$, such that the repeated orthonormal sequence $\{e_{j \mod d} \}_{j=1}^N$ minimizes $E_p$ over all size $N$ systems of unit vectors (with value $E_p(\mathbf{A})=d(k^2-k)+2k$) for $p<p_0$ and the minimum value of $E_p(\mathbf{A})$  satisfies $E_p(\mathbf{A})<d(k^2-k)+2k$ when $p>p_0$. Further $p_0=p_0(k)$ satisfies $p_0(k)\rightarrow 2$ as $k\rightarrow \infty$.
\end{conj}

Theorem \ref{thm:dim-2}, can be interpreted as an improvement to Theorem~\ref{thm:d+1} from the previous section, and partial progress towards the above conjecture. The parameter $p_0=1.3$ in Theorem \ref{thm:dim-2} is the same for all values of $k$. It might be interesting to find an elementary argument which showed that $p_0(k)\rightarrow 2$ in dimension $2$, let alone generally. 

Theorem~\ref{thm:d+1} and Conjecture~\ref{conj:onb} appear to be examples of a more general phenomenon. The direct analogs/extensions of orthonormal bases are regular simplices in the non-projective setting and maximal ETFs in projective spaces. In the first case, the analogous potential function is $f^{\Delta}_{d,p}(t)=|t+\frac 1 d|^p$. In the second case, one can view the potential function $f(t)=|t|^{p}$ as an instance of the more general function $f_{\alpha,p}(t)=|t^2-\alpha^2|^{p}$, and the orthonormal basis as an example of an equiangular tight frame ($|\langle x,y \rangle|=\alpha$, for $x\neq y$ with coherence $\alpha=0$). The problem of minimizing the energies $E^{\Delta}_{d,p}$ and $E_{\alpha,p}$, associated with $f^{\Delta}_{d,p}$ and $f_{\alpha,p}$, respectively, for $p$ close to $0$ might be expected to pick out repeated regular simplices and repeated ETFs with coherence $\alpha$. We conjecture generally:

\begin{conj} \label{conj:etf} 

\begin{enumerate}
\item[(i)] Let $\mathbf{A}\in\R^{N\times N}$ be the Gram matrix of $N=m+k(d+1)$ points given in $\mathbb{S}^{d-1}$, with $1\leq m<d+1$, $d\geq 2$, and let $\{\varphi_i\}_{i=1}^l\subset \R^{d}$ be the maximal regular simplex in $\mathbb{S}^{d-1}$. Then there is a value of $p_0>0$, such that the repeated regular simplex $\{\varphi_{j \mod l} \}_{j=1}^N$ minimizes $E^{\Delta}_{d,p}$ over all size $N$ systems of unit vectors for $p<p_0$.

\item[(ii)] Let $\mathbf{A}\in\R^{N\times N}$ be the Gram matrix of $N=m+kl$ points given in $\mathbb{S}^{d-1}$, with $1\leq m<l$, $d\geq 2$, and $l$ the size of an ETF with coherence $\alpha$, $\{\varphi_i\}_{i=1}^l\subset \R^{d}$. Then there is a value of $p_0>0$, such that the repeated ETF sequence $\{\varphi_{j \mod l} \}_{j=1}^N$ minimizes $E_{\alpha,p}$ over all size $N$ systems of unit vectors for $p<p_0$.

\end{enumerate}
\end{conj}

Conjectures \ref{conj:onb} and \ref{conj:etf} emphasize the possibility of repeated configurations minimizing $p$-frame or $p$-frame-type energies among sets of $N$ points for all large enough $N$. This property may be seen as a strong version of stability of an optimizing set. 

The collections of unit vectors $\Phi=\{\varphi_i\}_{i=1}^l\subset \R^{d}$ that, for all $N\geq l$, have a repeated set $\{\varphi_{j \mod l} \}_{j=1}^N$ minimizing the energy defined by the potential function $f$ among all $N\times N$ Gram matrices $\mathbf{A}$, also satisfy that the uniform distribution over $\Phi$ must solve the problem,
\begin{equation}\label{eq:conten}\min\limits_{\mu\in\mathcal{P}(\S^{d-1})} I_f(\mu)=\min\limits_{\mu\in\mathcal{P}(\S^{d-1})} \int_{\S^{d-1}}\int_{\S^{d-1}} f(\langle x, y\rangle)d\mu(x)d\mu(y),\end{equation}
 where  $\mathcal{P}(\S^{d-1})$ are Borel probability measures, $\mu(\S^{d-1})=1$. For the case of $p$-frame energies, {\it tight designs} are examples of configurations such that uniform distributions over them minimize $I_f(\mu)$ over ranges of $p$ \cite{bil19a}. This behavior is slightly different from that conjectured above, as repeated tight designs of size $l$ can only be expected to minimize the discrete energies when $N=kl$, generally. 

The existence of a $p_{0}$ in these conjectures might be expected in connection with ideas from the field of {\it compressed sensing}. One should expect that, as $p\rightarrow 0$, the solution to this minimization problem is a repeated ETF (as upon vectorizing the Gram matrix, and considering the difference, the sparsest difference arises this way), but Conjecture~\ref{conj:etf} strengthens this to say that the solution for $p$ sufficiently small is also a repeated ETF. 

In connection with these observations we collect further support for the above conjectures for maximal ETFs and the regular simplex, by showing they minimize the associated continuous energies (\ref{eq:conten}). The second part of the below theorem holds also with the coherence replaced with the corresponding value on a complex maximal ETF, these being known alternatively as {\it symmetric informationally complete positive operator-valued measures} (SIC-POVMs, see \cite{fuchs17} for more details). We give the proof only in the real case, but the same type of argument applies in the complex case.

\begin{theorem} \label{thm:cont-energy} The following statements hold:

\begin{enumerate}

\item[(i)] The uniform distribution over the vertices of a regular simplex $\{\varphi_i\}_{i=1}^{d+1}\subset \S^{d-1}$ minimizes the continuous energy $I^{\Delta}_{d, p}$ for $f^{\Delta}_{d,p}(t)=|t+\frac{1}{d}|^p$ for all $p\in(0,2]$.

\item[(ii)] Whenever a maximal ETF exists, $\{\varphi_i\}_{i=1}^M\subset \S^{d-1}$ (with coherence $\alpha^2=\frac{1}{d+2}$), the uniform distribution over its points minimizes the continuous energy $I_{\alpha,p}$ for $f_{\alpha,p}(t)=|t^2-\alpha^2|^p$ for all $p\in(0,2]$.
\end{enumerate}

\begin{proof}

Note that the inequalities

\begin{equation}\label{eq:ineqs} \left|\frac{t+\frac{1}{d}}{1+\frac{1}{d}}\right|^p \geq \left|\frac{t+\frac{1}{d}}{1+\frac{1}{d}}\right|^2\ \text{ and }\ \left|\frac{t^2-\frac{1}{d+2}}{1-\frac{1}{d+2}}\right|^p \geq \left|\frac{t^2-\frac{1}{d+2}}{1-\frac{1}{d+2}}\right|^2, \end{equation}
hold for $t\in[-1,1]$. The first inequality implies $I^{\Delta}_{d,p}\geq (1+\frac 1 d)^{p-2} I^{\Delta}_{d,2}$ and the equality holds for the uniform distribution over the regular simplex. The second inequality implies $I_{\alpha,p}\geq (1-\frac 1 {d+2})^{p-2} I_{\alpha,2}$ and the equality holds for the uniform distribution over a maximal ETF. It is sufficient then to prove the theorem for $p=2$.

For the proof in the case $p=2$, we use the notion of {\it positive definite functions} on unit spheres. A function of inner products is positive definite on $\S^{d-1}$ if for any collection of points $\{ x_{i}\}_{i=1}^N \subset \S^{d-1}$ and sequence of complex numbers $\{c_{i}\}_{i=1}^N\subset\C$ it is true that
$$\sum\limits_{i,j=1}^{N} c_i\overline{c_j}f(\langle x_i, x_j \rangle)\geq 0.$$
This condition implies that for any positive definite $f$ and any measure $\mu$ on $\S^{d-1}$,
$$I_f(\mu)=\int_{\S^{d-1}}\int_{\S^{d-1}} f(\langle x, y\rangle)d\mu(x)d\mu(y)\geq 0.$$

Positive definite functions on spheres were characterized by Schoenberg \cite{sch41}. In particular, $t$, $t^2-\frac 1 d$, and $t^4-\frac 6 {d+4} t^2+ \frac 3 {(d+2)(d+4)}$ (Gegenbauer polynomials of degrees 1, 2, 4) are positive definite functions on $\S^{d-1}$.

Since $(t+\frac 1 d)^2= (t^2-\frac 1 d) + \frac 2 d t + \frac {d+1} {d^2}$,
$$I^{\Delta}_{d,2} (\mu) \geq \frac {d+1} {d^2}.$$
It is easy to check that for the uniform distribution over the regular simplex, $I^{\Delta}_{d,2}$ is precisely $\frac {d+1} {d^2}$.

Using $$\left(t^2-\frac 1 {d+2}\right)^2 = \left(t^4-\frac 6 {d+4} t^2+ \frac 3 {(d+2)(d+4)}\right) + \frac {4(d+1)} {(d+2)(d+4)} \left(t^2-\frac 1 d\right) + \frac {2(d+1)} {d(d+2)^2},$$
we conclude that
$$I_{\alpha,2} (\mu) \geq \frac {2(d+1)} {d(d+2)^2}.$$
Again it is easy to check that the equality is attained on the uniform distribution over the maximal ETF.
\end{proof}

\end{theorem}

Using the design conditions, one can also show that the configurations from Theorem \ref{thm:cont-energy} are unique minimizers (up to the uniqueness of maximal ETFs in the second case) for the corresponding energies when $p\in(0,2)$, similarly to how it is done in \cite{bil19a} for the general case of tight designs and $p$-frame energies. The second part of Theorem \ref{thm:cont-energy} relies on the existence of maximal ETFs. It would be interesting to study minimizers in the case when this existence is not granted. We expect that properties of minimizing measures are similar to those for $p$-frame energies \cite{bil19b}

\section{Acknowledgments}

The authors would like to thank Wei-Hsuan Yu for fruitful discussions, Boris Bukh and Chris Cox for sharing the manuscript and discussing their results, Kasso Okoudjou for talking about his and his coauthors' research in this direction and informing about their conjectures, and anonymous referees for their help with the text. This paper is based upon work supported by the National Science Foundation under Grant No. DMS-1439786 while the first author was in residence at the Institute for Computational and Experimental Research in Mathematics in Providence, RI, during the Spring 2018 semester. The second author was supported in part by grant from the US National Science Foundation, DMS-1600693.

\bibliographystyle{amsalpha}

\end{document}